\documentclass{amsart}
\usepackage{pxfonts}
\newtheorem{theorem}{Theorem}
\newtheorem{lemma}{Lemma}
\newtheorem{proposition}{Proposition}
\newtheorem{example}{Example}
\newtheorem{definition}{Definition}
\begin{document}
\title{\Large Realized Cumulants for Martingales}
\author{%
\large   M. Fukasawa
}
\author{\large K. Matsushita}
\begin{abstract}
{\normalsize
Generalizing the realized variance,
the realized skewness (Neuberger, 2012) and the realized kurtosis
 (Bae and Lee, 2020), we construct realized cumulants with the so-called 
aggregation property.
They are unbiased statistics of the cumulants of a martingale marginal
 based on sub-period increments of the martingale and its lower-order
 conditional cumulant processes. Our key finding is a relation between
 the aggregation
 property and the complete Bell polynomials.
For an application we give an alternative proof and an extension 
of a cumulant recursion
formula recently obtained by Lacoin et al.~(2019) and Friz et al.~(2020).
}
\end{abstract}

\maketitle

\section{Introduction}
A square-integrable martingale $M = \{M_t\}_{t \in [0,T]}$ satisfies
\begin{equation}
    \label{eq:intro}
    E[(M_T-M_0)^n] = E\left[ \sum_{j=1}^{N} (M_{t_j}-M_{t_{j-1}})^n
		      \right]
\end{equation}
for $n=2$ with an arbitrary time partition  $0=t_0<\dots<t_N=T$.
On one hand this is a crucial property to develop a theory of
martingale (see e.g., \cite{CE}) and on the other hand, 
after the seminal work~\cite{ABDL},
this has played an important role in
financial econometrics, where the sum of the squared increments is
called the realized variance and has been used as an accurate measure of
the
volatility of an asset price modelled by $M$.
Unfortunately (\ref{eq:intro}) does not hold for $n\geq 3$, which
causes a difficulty in estimating higher
order moments based on high frequency data; see~\cite{N} for the details.

For a Borel function $g$ on $\mathbb{R}^d$ and an $\mathbb{R}^d$-valued adapted process $\mathbb{X}$ on a filtered probability
space $(\Omega,\mathcal{F},P,\{\mathcal{F}_t\}_{t \in [0,T]})$,
the aggregation property introduced by Neuberger~\cite{N} refers to
that
    \begin{equation}
        \label{eq:Aggregation Property}
        E\left[ g(\mathbb{X}_u-\mathbb{X}_s) |\mathcal{F}_s
	 \right] = E\left[ g(\mathbb{X}_u-\mathbb{X}_t)
	  |\mathcal{F}_s \right] + E[g(\mathbb{X}_t-\mathbb{X}_s)|\mathcal{F}_s]
    \end{equation}
    for any $s \le t \le u \leq T$. 
The identity (\ref{eq:intro}) with $n=2$ can be seen as a consequence of
the aggregation property of $g(x) = x^2$ and $\mathbb{X}=M$.
An interesting finding by Neuberger~\cite{N} is that
the aggregation property is met by
$g(x,y) = x^3 + 3xy$ and $\mathbb{X} = (M,M^{(2)})$ for any 
$L^3$-martingale $M$, where   $M^{(n)}_t =
E[(M_T-M_t)^n|\mathcal{F}_t]$.  Observing that this implies
\begin{equation*}
 E[(M_T-M_0)^3] = E[g(\mathbb{X}_T-\mathbb{X}_0)] = E\left[
\sum_{j=1}^N g(\mathbb{X}_{t_j}-\mathbb{X}_{t_{j-1}})
\right],
\end{equation*}
Neuberger~\cite{N} named the high-frequency statistic
\begin{equation*}
 \sum_{j=1}^N g(\mathbb{X}_{t_j}-\mathbb{X}_{t_{j-1}})
= \sum_{j=1}^N\left\{ (M_{t_j}-M_{t_{j-1}})^3 +
	              3(M_{t_j}-M_{t_{j-1}})(M^{(2)}_{t_j}-M^{(2)}_{t_{j-1}})\right\}
\end{equation*}
the realized skewness, and conducted an econometric analysis with it.
 Bae and Lee~\cite{BL} further extended the idea to construct
the realized kurtosis, by 
finding that the aggregation property is met by
$g(x,y,z) = x^4+ 6x^2y + 3y^2 + 4xz$ and
$\mathbb{X} = (M,M^{(2)},M^{(3)})$.
Neuberger~\cite{N} and
 Bae and Lee~\cite{BL}'s approach to find those polynomials is rather
 brute force. Indeed,  they showed also that there is no other
 analytic function (up to linear combinations) $g$ satisfying the aggregation property with
$\mathbb{X} = (M,M^{(2)},M^{(3)})$.
An extension to the higher order cumulants (or moments) has been left open.

The aim of this paper is to unveil this mysterious property.
Our finding is that
\begin{equation*}
 g_n(x_1,\dots,x_n) = B_{n+1}(x_1,\dots,x_n,0)
\end{equation*}
satisfies the aggregation property with
\begin{equation*}
 \mathbb{X}^{(n)} = (X^{(1)},\dots,X^{(n)})
\end{equation*}
for any $n \in \mathbb{N}$, 
where $B_{n+1}$ is the $(n+1)$-th complete Bell polynomial and
$X^{(i)}$ is the $i$th conditional cumulant process of an $L^n$
integrable
 random variable $X$
(see Section~4 for the details).
When $X = M_T$, then $\mathbb{X}^{(3)} = (M,M^{(2)},M^{(3)})$ and we have
\begin{equation}\label{intromain}
X^{(n+1)}_t = E\left[\sum_{j=1}^N
		g_n(\mathbb{X}^{(n)}_{t_j} -
		\mathbb{X}^{(n)}_{t_{j-1}}) \bigg|
		\mathcal{F}_t\right]
\end{equation} 
for any partition
 $t=t_0<\dots<t_N=T$ of $[t,T]$.
The cases $n=1$, $2$ and $3$ correspond to the realized variance,
Neuberger's realized skewness and Bae and Lee's realized kurtosis
respectively. 

Taking high-frequency limit in (\ref{intromain}), we have
\begin{equation*}
 X^{(n+1)}_t = E\left[
 \sum_{s \in (t,T]} g_n(\Delta \mathbb{X}^{(n)}_s) + 
\frac{1}{2}\sum_{j=1}^n \binom{n+1}{j}\int_{t}^T \mathrm{d}
\langle X^{(n+1-j),c}, X^{(j),c} \rangle
\bigg|\mathcal{F}_t
\right],
\end{equation*}
where $X^{(j),c}$ is the continuous local martingale part of the
semimartingale $X^{(j)}$. This extends a cumulant recursion formula 
recently obtained by Lacoin et al.~\cite{LRV} and Friz et
al.~\cite{FGR}.

After recalling the complete Bell polynomials and their properties in
Section~2, we introduce conditional cumulants in Section~3.
Then, we give our main results in Section~4.
The application to the cumulant recursion
formula is given in Section~5.

\section{The complete Bell polynomials}
Here, we recall the complete Bell polynomials and some of their
properties we use in this work.
This section does not contain any new results but we include proofs for
the readers' convenience.

\begin{definition}[The complete Bell polynomials]
Let  $n \in \mathbb{N}$.
For $(x_1,\dots,x_n) \in \mathbb{R}^n$,
     the $n$th
 complete Bell polynomial $B_n(x_1,\dots,x_n)$ is defined by
    \begin{equation}
        B_n(x_1,\dots,x_n) = \frac{\partial^{n}}{\partial
				  z^{n}} \exp\left( \sum_{i=1}^{n} x_i
					      \frac{z^i}{i!} \right)
				 \bigg|_{z=0}
    \end{equation}
    with $B_0=1$.
\end{definition}

\begin{example}\label{ex1}
Here are some examples of the complete Bell poynomials.
    \begin{equation*}
        \begin{split}
& B_1(x_1) = x_1,\\
& B_2(x_1,x_2) = x_1^2 + x_2,\\
&B_3(x_1,x_2,x_3) = x_1^3 + 3x_1x_2 + x_3,\\
&B_4(x_1,x_2,x_3,x_4) = x_1^4 + 6x_1^2x_2 + 4x_1x_3 + 3x_2^2 + x_4,\\
            &B_n(x_1,0,\dots,0) = \left. \frac{\partial^{n}}{\partial
	 z^{n}} \sum_{k=0}^{\infty} \frac{1}{k!} \left( x_1 z \right)^k
	 \right|_{z=0} = \left. \frac{\partial^{n}}{\partial z^{n}}
	 \frac{1}{n!} \left( x_1 z \right)^n \right|_{z=0} = (x_1)^n, \\
            &B_n(0,\dots,0,x_n) = \left. \frac{\partial^{n}}{\partial
	 z^{n}} \sum_{k=0}^{\infty} \frac{1}{k!} \left( x_n
	 \frac{z^n}{n!} \right)^k \right|_{z=0} =
	 \left. \frac{\partial^{n}}{\partial z^{n}} \left( x_n
	 \frac{z^n}{n!} \right) \right|_{z=0} = x_n.
        \end{split}
    \end{equation*}    
\end{example}

\begin{proposition}[The generating function]\label{prop:gen}
If one of 
 $$\sum_{n=1}^{\infty}x_n \frac{z^n}{n!}
\ \ \text{and} \ \ \sum_{n=0}^{\infty} B_n(x_1,\dots,x_n) \frac{z^n}{n!}$$
is absolutely convergent on a neighbourhood of $z = 0$, then so is the
 other and
\begin{equation}\label{gen}
 \exp\left(\sum_{n=1}^{\infty}x_n \frac{z^n}{n!}\right)
= \sum_{n=0}^{\infty} B_n(x_1,\dots,x_n) \frac{z^n}{n!}
\end{equation}
on a neighborhood of $z = 0$.
\end{proposition}
\begin{proof}
If, for example, the first series is absolutely convergent on a
 neighborhood of $z = 0$, then 
the left hand side of (\ref{gen}) is analytic on the neighborhood and 
so, admits
 the Taylor expansion around $z=0$. The coefficients are determined as
\begin{equation*}
\frac{\partial^n}{\partial z^n}
\sum_{k=0}^{\infty}\frac{1}{k!}\left(\sum_{i=1}^{\infty}x_i\frac{z^i}{i!}\right)^k
\bigg|_{z=0}
= 
\frac{\partial^n}{\partial z^n}
\sum_{k=0}^{\infty}\frac{1}{k!}\left(\sum_{i=1}^{n}x_i\frac{z^i}{i!}\right)^k
\bigg|_{z=0}
= B_n(x_1,\dots,x_n).
\end{equation*}
\end{proof}

\begin{proposition}[A binomial type relation]
    \label{prop:bell}
    Let $n \in \mathbb{N}$ and $(x_1,\dots x_n), 
(y_1,\dots,y_n) \in \mathbb{R}^n$. Then,
    \begin{equation}
        \label{eq:bell}
        B_n(x_1+y_1,\dots,x_n+y_n) = \sum_{j=0}^{n} \binom{n}{j}
        B_{n-j}(x_1,\dots,x_{n-j}) B_j(y_1,\dots,y_j). 
    \end{equation}
\end{proposition}
\begin{proof}
    By the Leibniz rule, we have
    \begin{equation*}
        \begin{split}
            B_n(x_1+y_1,\dots,x_n+y_n)             &=
	 \left. \frac{\partial^{n}}{\partial t^{n}} \exp\left(
	  \sum_{i=1}^{n} (x_i+y_i) \frac{t^i}{i!} \right)
	  \right|_{t=0} \\
            &= \left. \left[ \sum_{j=0}^{n} \binom{n}{j}
	  \frac{\partial^{n-j}}{\partial t^{n-j}} \exp\left(
	  \sum_{i=1}^{n} x_i \frac{t^i}{i!} \right)
	  \frac{\partial^{j}}{\partial t^{j}} \exp\left(
	  \sum_{i=1}^{n} y_i \frac{t^i}{i!} \right) \right]
	  \right|_{t=0} \\
            &= \sum_{j=0}^{n} \binom{n}{j} B_{n-j}(x_1,\dots,x_{n-j})
	 B_j(y_1,\dots,y_j).
        \end{split}
    \end{equation*}
\end{proof}

\begin{example}
By (\ref{eq:bell}) and the last identity of Example~\ref{ex1},
for $n \geq 2$,
 \begin{equation}\label{dec}
\begin{split}
  B_n(x_1,\dots,x_{n-1},x_n)
&= B_n(x_1,\dots,x_{n-1},0) + B_n(0,\dots,0,x_n) \\
&=  B_n(x_1,\dots,x_{n-1},0) + x_n.
\end{split}
 \end{equation}
\end{example}

\begin{proposition}\label{ident}
     Let $n \in \mathbb{N}$ and $(x_1,\dots x_n), 
(y_1,\dots,y_n) \in \mathbb{R}^n$. If
\begin{equation*}
 B_k(x_1,\dots,x_k) = B_k(y_1,\dots,y_k)
\end{equation*}
for all $k \leq n$. Then 
$(x_1,\dots x_n)=
(y_1,\dots,y_n)$.
\end{proposition}
\begin{proof}
The result follows by induction. Indeed,
 $x_1 = B_1(x_1) = B_1(y_1) = y_1$. If $x_i = y_i$ for all $i \leq k-1$,
 then, by (\ref{dec}),
\begin{equation*}
\begin{split}
 x_k &= B_k(x_1,\dots,x_{k-1},x_k)- B_k(x_1,\dots,x_{k-1},0) \\
&= B_k(y_1,\dots,y_{k-1},y_k)- B_k(y_1,\dots,y_{k-1},0)  = y_k.
\end{split}
\end{equation*}
\end{proof}
 \begin{proposition}\label{prop:mandc1}
   Let $p \geq 1$ and $X$  be an $L^p$ 
integrable random variable on a probability space
$(\Omega,\mathcal{F},P)$.
For a positive integer 
$n \leq p$,
define the $n$th cumulant $\kappa_n$
of $X$ by
\begin{equation}\label{def-cum1}
        \kappa_n  = (-\sqrt{-1})^n \frac{\partial^{n}}{\partial
     z^n}  \log E[e^{\sqrt{-1}z X}] \bigg|_{z=0}.
\end{equation}
Then, 
    \begin{equation}\label{momentformula1}
       E[X^n] = B_n(\kappa_1,\dots,\kappa_n).
    \end{equation}
 \end{proposition}
\begin{proof}
First note that
    \begin{equation*}
        E[X^n]= (-\sqrt{-1})^n \frac{\partial^{n}}{\partial z^{n}}
	              E[e^{\sqrt{-1}z X}] \bigg|_{z=0}.
    \end{equation*}
From this and (\ref{def-cum1}), it is clear that $\kappa_n$ is a polynomial
 of $E[X^m]$, $m \leq n$. 
Note also that 
$L^n \ni X \mapsto E[X^m] \in \mathbb{R} $ is continuous for $m
 \leq n$.
Therefore $L^n \ni X \mapsto \kappa_m \in \mathbb{R} $ is continuous
 for $m \leq n$.
Since $X$ can be approximated by bounded random variables in $L^n$,
in order to show (\ref{momentformula1}),
it is then sufficient to consider bounded $X$.
For a bounded random variable $X$, $z \mapsto E[e^{zX}]$ is analytic and so,
\begin{equation*}
 \sum_{n=0}^{\infty}E[X^n] \frac{z^n}{n!} =
E[e^{zX}]
= \exp\left(
\sum_{n=1}^{\infty}\kappa_n \frac{z^n}{n!}
\right)
\end{equation*}
on a neighborhood of $z = 0$,
which implies (\ref{momentformula1}) in the light of (\ref{gen}).
\end{proof}

See \cite{P} for combinatorial aspects of the Bell polynomials and cumulants.

\section{The conditional cumulants}

Here we introduce conditional cumulants that play a key role
in this work.
   Let $p \geq 2$ and $X$  be an $L^p$ 
integrable random variable on a probability space
$(\Omega,\mathcal{F},P)$.
Let $\mathcal{G}$ be  a sub $\sigma$-algebra  of $\mathcal{F}$.
If there exists a regular conditional probability measure given
$\mathcal{G}$,
then it is natural to 
define the $n$th conditional cumulant $X^{(n)}_\mathcal{G}$
of $X$ given $\mathcal{G}$ by
\begin{equation}\label{def-cum}
        X^{(n)}_\mathcal{G}  = (-\sqrt{-1})^n \frac{\partial^{n}}{\partial
     z^n}  \log E_\mathcal{G}[e^{\sqrt{-1}z X}] \bigg|_{z=0}
\end{equation}
for a positive integer $n \leq p$, 
where $E_\mathcal{G}$ denotes the expectation with respect to the
regular conditional probability measure.
Then, it is clear from Proposition~\ref{prop:mandc1} that
    \begin{equation}\label{momentformula}
       E[X^n|\mathcal{G}] = B_n(X^{(1)}_\mathcal{G},\dots,X^{(n)}_\mathcal{G}).
    \end{equation}
However in general, a regular conditional probability measure might not
exist and there might be no event of probability one on which
$z \mapsto E[e^{\sqrt{-1}zX}|\mathcal{G}]$ is differentiable,
since the conditional expectation is defined only up to a null set for
each $z$.
We therefore take another route. In light of Proposition~\ref{ident}, 
we use (\ref{momentformula}) as the defining property of conditional
cumulants.
\begin{proposition}\label{prop:mandc}
 Define $X^{(i)}_\mathcal{G}$, $i=1,\dots,[p]$ by
\begin{equation*}
 \begin{split}
&  X^{(1)}_\mathcal{G} = E[X|\mathcal{G}],\\
&X^{(n)}_\mathcal{G} = E[X^n|\mathcal{G}]-B_n(X^{(1)}_\mathcal{G},\dots,X^{(n-1)}_\mathcal{G},0)  \ \ \text{ for } n \geq 2.
 \end{split}
\end{equation*}
Then, $X^{(i)}_\mathcal{G} \in L^{p/i}$ and (\ref{momentformula}) holds
 for $n \leq p$.
Further, $L^p \ni X \mapsto X^{(i)}_\mathcal{G} \in L^{p/i}$ is continuous.
\end{proposition}
\begin{proof}
It is clear that $X^{(1)}_\mathcal{G} \in L^p$ and that
$L^p \ni X \mapsto X^{(1)}_\mathcal{G} \in L^p$ is continuous.
By Lemma~\ref{conti} below, it follows by induction that
$X^{(i)}_\mathcal{G} \in L^{p/i}$ and
that  $L^p \ni X \mapsto X^{(i)}_\mathcal{G} \in L^{p/i}$ is continuous.
(\ref{momentformula}) is clear from (\ref{dec}).\end{proof}

\begin{lemma}\label{conti}
Let $n \in \mathbb{N}$ and
 $p \geq n$. For $(X_1,X_2,\dots,X_n) \in L^{p} \times L^{p/2}\times
 \dots \times L^{p/n}$, 
$B_n(X_1,X_2,\dots,X_n) \in L^{p/n}$. Further,
the map
$$L^{p} \times L^{p/2}\times \dots \times L^{p/n} \ni
(X_1,X_2,\dots,X_n) \mapsto B_n(X_1,X_2,\dots,X_n) \in L^{p/n}$$
is continuous.
\end{lemma}
\begin{proof}
$B_n(x_1,\dots,x_n)$ is a linear combination of terms of the form
$ \prod_{j=1}^k x_{i_j}$,
where $i_j \in \{1,\dots,n\}$ with $\sum_{j=1}^k i_j = n$. 
Therefore it suffices to show that
$\prod_{j=1}^k X_j \in L^{p/n}$ 
for $(X_1,\dots,X_k) \in  L^{p/i_1} \times \dots \times L^{p/i_k}$
and that
\begin{equation*}
 L^{p/i_1} \times \dots \times L^{p/i_k} \ni
(X_1,\dots,X_k) \mapsto \prod_{j=1}^k X_j \in L^{p/n}
\end{equation*}
is continuous. By the H\"older inequality,
\begin{equation*}
 \left\| \prod_{j=1}^k |X_j|^{p/n} \right\|_1
\leq \prod_{j=1}^k \||X_j|^{p/n}\|_{n/i_j}
= \prod_{j=1}^k \|X_j\|_{p/i_j}^{p/n},
\end{equation*}
which implies the result.
\end{proof}

\section{The aggregation property of the conditional cumulant processes}
Here we give our main results.
Let $T>0$ be a fixed constant and 
$(\Omega,\mathcal{F}, P, \{\mathcal{F}_t\}_{t \in [0,T]})$ be a
filtered probability space satisfying the usual conditions.
For a process $F = \{F_t\}$ and $s\leq t$, $F_{s,t}$ denotes $F_t-F_s$.
Let $p \geq 2$ and $X \in L^p$. 
For $n \leq p$, define the $n$th conditional
cumulant process $X^{(n)} = \{X^{(n)}_t\}$, where $X^{(n)}_t =
X^{(n)}_{\mathcal{F}_t}$ defined as in Proposition~\ref{prop:mandc} 
with $\mathcal{G} = \mathcal{F}_t$. This construction allows to take a
cadlag version.
Let $\mathbb{X}^{(n)} = (X^{(1)},\dots,X^{(n)})$.

\begin{lemma}\label{keylem}
 For any positive  integer $n \leq p$ and for any $t \leq u$,
\begin{equation*}
 E[B_n(\mathbb{X}^{(n)}_{t,u})|\mathcal{F}_t] = 0.
\end{equation*}
\end{lemma}
\begin{proof}
 By  Proposition~\ref{prop:mandc} and Lemma~\ref{conti},
\begin{equation*}
 L^p \ni X \mapsto  E[B_n(\mathbb{X}^{(n)}_{t,u})|\mathcal{F}_t]
  \in L^{p/n}
\end{equation*}
is continuous. Therefore it suffices to consider  $X \in L^\infty$
as $L^\infty$ is dense in $L^p$.
If $X \in L^\infty$, then by
Proposition~\ref{prop:gen} and (\ref{momentformula}), we have
\begin{equation*}
 \frac{1}{E[e^{zX}|\mathcal{F}_t]} =
\left(\sum_{n=0}^\infty E[X^n|\mathcal{F}_t]\frac{z^n}{n!}\right)^{-1} =  
\exp\left(-\sum_{n=1}^{\infty}X^{(n)}_t \frac{z^n}{n!}\right)
= \sum_{n=0}^\infty
B_n(-\mathbb{X}^{(n)}_t)\frac{z^n}{n!}
\end{equation*}
on a neighborhood of $z = 0$. This implies
\begin{equation*}
 1 = \left(\sum_{n=0}^\infty E[X^n|\mathcal{F}_t]\frac{z^n}{n!}\right)
\left(
\sum_{n=0}^\infty
B_n(-\mathbb{X}^{(n)}_t)\frac{z^n}{n!}
\right) = \sum_{n=0}^\infty \sum_{j=0}^n
\binom{n}{j}E[X^{n-j}|\mathcal{F}_t]
B_j(-\mathbb{X}^{(j)}_t)\frac{z^n}{n!}
\end{equation*}
and so for $n \geq 1$,
\begin{equation*}
 0 = \sum_{j=0}^n
\binom{n}{j}E[X^{n-j}|\mathcal{F}_t]
B_j(-\mathbb{X}^{(j)}_t) = 
E\left[\sum_{j=0}^n
\binom{n}{j}E[X^{n-j}|\mathcal{F}_u]
B_j(-\mathbb{X}^{(j)}_t) \bigg| \mathcal{F}_t \right].
\end{equation*}
The right hand side coincides with
$E[B_n(\mathbb{X}^{(n)}_u - \mathbb{X}^{(n)}_t)|\mathcal{F}_t]$
by Proposition~\ref{prop:bell} and (\ref{momentformula}).
\end{proof}

Now we give the main result of this paper.

\begin{theorem}\label{main}
For a positive integer $n \leq p-1$,
 define $g_n : \mathbb{R}^{n} \to \mathbb{R}$ by
\begin{equation*}
 g_n(x_1,\dots,x_n) = B_{n+1}(x_1,\dots,x_n,0).
\end{equation*}
Then, 
\begin{equation}\label{rcum}
E[g_n(\mathbb{X}^{(n)}_{t,u})|\mathcal{F}_t] =
 - E[X^{(n+1)}_{t,u}|\mathcal{F}_t]
\end{equation}
for any $t \leq u \leq T$.
In particular, $g_n$ satisfies the aggregation property with
 $\mathbb{X}^{(n)}$, that is,
\begin{equation}\label{agg}
 E[g_n(\mathbb{X}^{(n)}_{s,u})|\mathcal{F}_s] = 
 E[g_n(\mathbb{X}^{(n)}_{s,t})|\mathcal{F}_s] + 
 E[g_n(\mathbb{X}^{(n)}_{t,u})|\mathcal{F}_s]
\end{equation}
for any $s \leq t \leq u \leq T$.
\end{theorem}
\begin{proof}
(\ref{rcum}) is clear from (\ref{dec}) and Lemma~\ref{keylem}.
(\ref{agg}) is clear from (\ref{rcum}).
\end{proof}

\begin{example}\label{mainex}
 Let $M = \{M_t\}_{t\in [0,T]}$ is an $L^p$ martingale. Taking $X =
 M_T$, we have $X^{(1)} = M$ and 
$X^{(n)}_T = 0$ for $n\geq 2$. As is well-known, the second cumulant is
 the variance and so, we have
$X^{(2)} = M^{(2)}$, where $M^{(n)}_t =
 E[(M_{t,T})^n|\mathcal{F}_t]$ as in Introduction.
For $n = 1$, 
since $g_n(x_1) = B_2(x_1,0) = x_1^2$, (\ref{rcum}) means
$
 E[(M_{t,u})^2|\mathcal{F}_t] = E[(M_{t,T})^2|\mathcal{F}_t] -  E[(M_{u,T})^2|\mathcal{F}_t]$.
For a partition $0 = t_0 < \dots < t_N = T$,
\begin{equation*}
 \sum_{j=1}^N g_1(\mathbb{X}^{(1)}_{t_{j-1},t_j}) = 
\sum_{j=1}^N (M_{t_{j-1},t_j})^2
\end{equation*}
is the so-called realized variance. The aggregation property implies
\begin{equation*}
E\left[\sum_{j=1}^N (M_{t_{j-1},t_j})^2\right]
= E\left[
 \sum_{j=1}^N g_1(\mathbb{X}^{(1)}_{t_{j-1},t_j}) \right]
= E[g_1(\mathbb{X}^{(1)}_{0,T})] = E[(M_{0,T})^2].
\end{equation*}
The third cumulant is known to be the centered third moment, that is,
$X^{(3)} = M^{(3)}$.
For $n=2$, $g_n(x,y) = B_3(x,y,0) = x^3 + 3xy$.
Therefore
\begin{equation}\label{def-rcum}
 \sum_{j=1}^N g_n(\mathbb{X}^{(n)}_{t_{j-1},t_j})
\end{equation}
with $n=2$ is Neuberger's realized skewness.
Its expectation is $-E[X^{(3)}_{0,T}] = E[X^{(3)}_0] = E[(M_{0,T})^3]$ 
by the aggregation property.
For $n=3$, $g_n(x,y,z) = B_4(x,y,z,0) = x^4+6x^2y + 4xz + 3y^2$.
Therefore (\ref{def-rcum})
with $n=3$ is Bae and Lee's realized kurtosis.
Its expectation is $-E[X^{(4)}_{0,T}] = E[X^{(4)}_0]$ 
by the aggregation property.\\
\end{example}

Based on Theorem~\ref{main} and Example~\ref{mainex}, we 
suggest to call the high-frequency statistic of the form (\ref{def-rcum})
the $(n+1)$-th realized cumulant of the martingale $X^{(1)}$. 
It is an unbiased estimator of the $(n+1)$-th cumulant $X^{(n+1)}_0$ 
of $X$ 
when $X$ is $\mathcal{F}_T$ measurable and
 $\mathcal{F}_0$ consists of null sets and their complements.
\\

We conclude this section with a brief explanation on how the realized
cumulants can be used in financial econometrics.
Suppose that  $X$ represents an asset price at time
$T$ and there is a market of call and put options with maturity $T$ 
written on the asset. The market prices of the options determine a
probability measure (called a risk neutral measure) under which 
each option price is expressed as the
expectation of its option payoff.
The moments and hence cumulants of $X$ under the risk neural measure
 at any time $t < T$ can be therefore computed from the option market
 prices at time $t$.
The $n$th realized cumulant is an ex-post value based on time series of
those conditional cumulant processes,
whose expectation under the risk neutral measure at time $0$
coincides with $X^{(n)}_0$.
Any systematic deviation of the $n$th realized cumulant value from
$X^{(n)}_0$ is due to the difference between the risk neutral measure
and the physical probability measure, and is interpreted as a premium.
See Neuberger~\cite{N} and the references therein for analyses of
variance premium and skewness premium.

\section{
Application to a cumulant recursion formula
}

It is clear from definition that $X^{(n)}$ is an $L^{p/n}$
semimartingale.
The high-frequency limit of the $(n+1)$-th realized cumulant on $[t,T]$ is
\begin{equation*}
 \sum_{s \in (t,T]} g_n(\Delta \mathbb{X}^{(n)}_s) + 
\frac{1}{2}\sum_{j=1}^n \binom{n+1}{j}
\langle X^{(n+1-j),c}, X^{(j),c} \rangle_{t,T},
\end{equation*}
where $X^{(j),c}$ is the continuous local martingale part of $X^{(j)}$.
Here, the convergence is in probability.
The second term comes from the fact that 
the quadratic terms contained in
$B_{n+1}(x_1,\dots,x_{n+1})$ are
\begin{equation*}
\frac{1}{2}\sum_{j=1}^{n}\binom{n+1}{j}x_{n+1-j}x_j
\end{equation*}
since
\begin{equation*}
\exp \left(\sum_{i=1}^\infty x_i \frac{z^i}{i!} \right)
= 1 + \sum_{i=1}^\infty x_i \frac{z^i}{i!} + 
\frac{1}{2}\left( \sum_{i=1}^\infty x_i \frac{z^i}{i!} \right)^2 + \dots.
\end{equation*}
If $X \in L^\infty$, then $X^{(n)}$ is an $L^q$ semimartingale for any $q$,
which ensures the high-frequency convergence is also in $L^1$. 
The aggregation property then implies
\begin{equation*}
E[g_n(\mathbb{X}^{(n)}_{t,T})|\mathcal{F}_t] = E\left[
 \sum_{s \in (t,T]} g_n(\Delta \mathbb{X}^{(n)}_s) + 
\frac{1}{2}\sum_{j=1}^n \binom{n+1}{j}
\langle X^{(n+1-j),c}, X^{(j),c} \rangle_{t,T}
\bigg|\mathcal{F}_t
\right].
\end{equation*}
In case $X$ is $\mathcal{F}_T$ measurable, we have $X^{(n+1)}_T = 0$ and
so, obtain a recursion formula
\begin{equation*}
 X^{(n+1)}_t = E\left[
 \sum_{s \in (t,T]} g_n(\Delta \mathbb{X}^{(n)}_s) + 
\frac{1}{2}\sum_{j=1}^n \binom{n+1}{j}
\langle X^{(n+1-j),c}, X^{(j),c} \rangle_{t,T}
\bigg|\mathcal{F}_t
\right].
\end{equation*}
Moreover if the filtration is continuous, then $\Delta
\mathbb{X}^{(n)} = 0$ and so,
\begin{equation*}
 X^{(n+1)}_t =\frac{1}{2}\sum_{j=1}^n \binom{n+1}{j} 
 E\left[
\langle X^{(n+1-j)}, X^{(j)} \rangle_{t,T}
\bigg|\mathcal{F}_t
\right]
\end{equation*}
or equivalently, for $Y^{(j)} = X^{(j)}/j!$,
\begin{equation*}
Y^{(n+1)} = \frac{1}{2}\sum_{j=1}^n Y^{(n+1-j)} \diamond Y^{(j)},
\end{equation*}
where $\diamond$ is the diamond operation introduced by Al\`os et
al.~\cite{AGR}. 
This last formula 
has been recently obtained by Lacoin et al.~\cite{LRV} in their study of
Quantum Field Theory.
The assumption of $X \in L^\infty$ can be relaxed; see Friz et
al.~\cite{FGR}.
See \cite{FGR} also for several nice applications of this recursion
formula.


\begin{thebibliography}{9}
\bibitem{AGR}
Al\`os, E., Gatheral, J. and 
Radoi$\check{\text{c}}$i\'c, R. (2020).
Exponentiation of conditional expectations under stochastic volatility.
{\it Quantitative Finance}, 20(1), 13-27.

\bibitem{ABDL}
Andersen, T.G., Bollerslev, T., Diebold, F.X. and Labys, P. (2001): The
	Distribution of Realized Exchange
 Rate Volatility. {\it Journal of the American Statistical Association}, 96,
 42-55.

\bibitem{BL}
Bae, K. and Lee, S. (2020).
Realized higher-order comoments. To appear in
{\it Quantitative Finance}.

\bibitem{CE}
Cohen, S.N and Elliott, R. (2015).
{\it Stochastic Calculus and Applications}. Springer.

\bibitem{FGR}
Friz, P.K., Gatheral, J. and Radoi$\check{\text{c}}$i\'c, R. (2020).
Forests, cumulants, martingales. arXiv,
arXiv:2002.01448v3.

\bibitem{LRV}
Lacoin,H.,  Rhodes R., and Vargas, V. (2019). A probabilistic
	approach of ultraviolet renormalisation in the boundary
	Sine-Gordon model. arXiv, arXiv:1903.01394.

\bibitem{N} Neuberger, A. (2012). 
Realized Skewness. {\it Rev. Financ. Stud.} 25 (11), 3423-3455.

\bibitem{P}
Peccati, G. and Taqqu, M.S. (2011).
{\it Wiener Chaos: Moments, Cumulants and Diagrams}.
Springer.
\end{thebibliography}
\end{document}